\documentclass[12pt]{article}

\usepackage{geometry} 
\usepackage{graphicx} 

\usepackage{booktabs} 
\usepackage{array} 
\usepackage{paralist} 
\usepackage{verbatim} 
\usepackage{subfig} 

\usepackage{amsthm}
\usepackage{amsmath}
\usepackage{amssymb}
\usepackage{graphicx}
\usepackage{amscd}
\usepackage[all, cmtip]{xy}
\usepackage{mathtools}

\newtheorem{thm}{Theorem}[section]
\newtheorem{lem}[thm]{Lemma}

\newtheorem{cor}[thm]{Corollary}
\newtheorem{remark}{Remark}[section]

\theoremstyle{defn}

\numberwithin{equation}{section}

\def \bt {\begin{tabular}{l}}
\def \et {\end{tabular}}

\begin{document}

\title{Local Diffeomorphisms in Symplectic Space and Hamiltonian Systems with Constraints}

\author{Konstantinos Kourliouros\footnote{
Avenida Trabalhador Sancarlense, 400 - Centro, S\~ao Carlos - SP, 13566-590, \newline 
email: k.kourliouros@gmail.com}\\
ICMC-USP}

\maketitle

\begin{abstract}
In this paper we obtain exact normal forms with functional invariants for local diffeomorphisms, under the action of the symplectomorphism group in the source space. Using these normal forms we obtain exact classification results for the first occurring singularities of Hamiltonian systems with one-sided constraints, a problem posed by R. B. Melrose in his studies of glancing hypersurfaces. 
\end{abstract}

\medskip 



%
%
%
%
%

\section{Introduction}

All objects in this paper are either smooth $C^{\infty}$ or analytic germs at the origin of $\mathbb{R}^{2n}$, unless otherwise stated. By a symplectic space we mean a pair $(\mathbb{R}^{2n},\omega)$ where $\omega$ is a symplectic form, i.e. a closed, non-degenerate 2-form. It is well known (Darboux theorem c.f. \cite{A3}) that all symplectic spaces are locally equivalent, that is, there exists local coordinates $(p,q):=(p_1,q_1,\cdots,p_n,q_n)$ such that $\omega$ can be reduced to the standard Darboux normal form:
\begin{equation}
\label{Darboux}
\omega=dp\wedge dq,
\end{equation}
where we denote $dp\wedge dq:=\sum_{i=1}^ndp_i\wedge dq_i$.

The main purpose is to give local classification results, under the action of symplectomorphisms of the Darboux normal form $\omega=(\ref{Darboux})$, for the following associated objects: 
\begin{itemize}
\item[-] Local diffeomorphisms, i.e. maps $\Phi:(\mathbb{R}^{2n},\omega)\rightarrow \mathbb{R}^{2n}$, $\det \Phi_*(0)\ne 0$: two diffeomorphisms $\Phi$, $\Phi'$ are equivalent if there exists a symplectomorphism $\Psi$ of $\omega$, $\Psi^*\omega=\omega$, such that $\Psi^*\Phi'=\Phi$.
\item[-] First occurring singularities of pairs $(f,H)$, where $f$ is a non-singular function and $H=\{h=0\}$ is a smooth hypersurface: two pairs $(f,H)$, $(f',H')$ are equivalent if there exists a symplectomorphism $\Psi$ of $\omega=(\ref{Darboux})$, $\Psi^*\omega=\omega$, such that $\Psi^*f'=f$, $\Psi(H')=H$.
\end{itemize} 
These objects, as will it become apparent in the text, are related in the following way: symplectic classification of diffeomorphisms $\Phi$ (Theorem \ref{thm-1}) depends on the symplectic classification of non-singular pairs $(f,H)$ (Lemma \ref{lem-1}) and on the structure of their isotropy group (Lemma \ref{lem-2}), a subgroup of the symplectomorphism group. On the other hand, symplectic classification of the first occurring singularities of pairs $(f,H)$ (Theorem \ref{thm-2}) depends, through a reduction process (Lemma \ref{lem-3}), on the symplectic classification of diffeomorphisms $\Phi$.

Pairs $(f,H)$ in symplectic space $(\mathbb{R}^{2n},\omega)$ can be identified with Hamiltonian systems with one-sided constraints, where $f$ is the Hamiltonian function and the hypersurface $H$ represents the set of one-side constraints (e.g. a boundary or an obstacle, lifted from the configuration space to the phase space of the system). The problem of their classification was posed by R. B. Melrose in \cite{M}, as a problem of considerable importance, in relation to boundary value problems of pseudo-differential operators, but it is also related to many other variational problems with constraints such as the classical Dirac problem (posed by P. M. Dirac in \cite{Dir}, see also \cite{A0}), the problem of bypassing an obstacle,  the billiard ball problem, and others, most of them translated nowadays in the language of Lagrangian singularities (by V. I. Arnol'd and his school c.f. \cite{A1}, \cite{A2}, \cite{A3}, \cite{A4} and references therein).

Symplectic classification of first occurring singularities of pairs $(f,H)$ is only well understood/known in the 2-dimensional case, c.f. \cite{K} in the analytic and recently \cite{Kir} in the smooth category, where the authors provide a relative analog of the classical isochore Morse lemma (c.f \cite{Ve} and also \cite{F}, \cite{G} for the analytic and \cite{CV} for the smooth case).

Recently an exact normal form for the first occurring singularities of triples $(\omega,f,H)$ (under the whole group of diffeomorphisms) has also been obtained in \cite{KM}, which holds in any dimension and in all the categories as well. This normal form though, despite the fact of being exact, it is only implicit, in the sense that all the functional invariants are hidden within the symplectic form $\omega$. The main motivation for writing this paper arose from the efforts to make this normal form explicit, i.e. to reduce $\omega$ back to its Darboux normal form (\ref{Darboux}) and obtain an exact normal form for the pair $(f,H)$, containing explicitly all the functional invariants of the classification problem. This dual normal form, which gives an answer to Melrose's original problem, is also ``better''  in terms of applications to boundary value problems, where the symplectic form is usually fixed to be the standard symplectic form of the phase space.

As it was mentioned before, the solution of Melrose's problem turned out to depend, through a reduction process, on the symplectic classification of diffeomorphisms. The corresponding results obtained here provide an infinite dimensional analog of the following classical fact from representation theory: the space of all linear symplectic structures is a homogeneous space diffeomorphic to the coset space $\mbox{GL}(2n)/\mbox{Sp}(2n)$. The dimension of this space, equal to $n(2n-1)$, constitutes the first term in the Poincar\'e series of the space of all symplectic structures (appearing also in \cite{D}), intrinsically associated to our classification problems.

The computation of the Poincar\'e series of Hamiltonian systems with constraints, (which lies in the circle of problems posed by Arnol'd in \cite{A6}), as well as the geometric description (without reference to coordinates) of the functional invariants obtained in this paper, are much more difficult problems (for $n\geq 2$) and will be treated in a subsequent paper. Furthermore, we do not discuss the more degenerate cases -singularity classes- but we remark that the methods provided in this paper can be used, with certain modifications, to attack these cases as well.

\section{Main Results}

For economy in the exposition, we will use throughout the paper the following:

\medskip

\noindent NOTATION.  We denote by $I^{\omega}_i$, $i=1,\cdots, 2n-1$, the nested sequence of ideals generated by the first $i$-Darboux coordinate functions of $\omega=(\ref{Darboux})$, but in reverse order, i.e.:
\[I^{\omega}_1=<q_1>,\]
\[I^{\omega}_2=<q_1,p_1>,\]
\[\vdots\]
\[I^{\omega}_{2n-2}=<q_1,p_1,\cdots,q_{n-1},p_{n-1}>,\]
\[I^{\omega}_{2n-1}=<q_1,p_1,\cdots,q_{n-1},p_{n-1},q_n>.\]

\subsection{Local Equivalence and Moduli of Diffeomorphisms in Symplectic Space}

\medskip

Consider a map $\Phi:(\mathbb{R}^{2n},\omega)\rightarrow \mathbb{R}^{2n}$, $\Phi(0)=0$, of maximal rank (a diffeomorphism), written in Darboux coordinates of $\omega=(\ref{Darboux})$ as:
\[\Phi(p,q)=(P_1(p,q),Q_1(p,q)\cdots,P_n(p,q),Q_n(p,q)),\]
where, up to renumeration, we may suppose that the following conditions always hold:
\begin{equation}
\label{con-2}
\partial_{p_i}P_i(0)\ne 0, \quad \partial_{q_i}Q_i(0)\ne 0, \quad i=1,\cdots,n.
\end{equation}
The problem is to find an exact normal form of $\Phi$ under symplectomorphisms of the Darboux normal form $\omega=(\ref{Darboux})$ in the source (no changes of coordinates in the target are allowed).

\begin{thm}
\label{thm-1}
Any diffeomorphism $\Phi$ in the symplectic space $(\mathbb{R}^{2n},\omega)$ can be reduced, by a symplectomorphism of  $\omega=(\ref{Darboux})$, to the normal form:
\begin{equation}
\label{nf-diffeo}
\Phi(p,q)=(p_1,\widetilde{Q}_1(p,q),p_2+\widetilde{P}_2(p,q),\widetilde{Q}_2(p,q),\cdots,p_n+\widetilde{P}_n(p,q),\widetilde{Q}_n(p,q)),
\end{equation}
where $\partial_{q_i}\widetilde{Q}_i(0)\ne 0$, $i=1,\cdots,n$ and the $(2n-1)$ functions of $2n$-variables $\{\widetilde{Q}_i(p,q)\}_{i=1}^{n}$, $\{\widetilde{P}_i(p,q)\}_{i=2}^{n}$, belong to the finitely generated ideals: 
\begin{equation}
\label{Q} 
\{\widetilde{Q}_i\in I^{\omega}_{2i-1}\}_{i=1}^n, \quad \{\widetilde{P}_i\in I^{\omega}_{2i-2}\}_{i=2}^n.
\end{equation}
These functions are functional invariants, i.e. they distinguish between non-equivalent diffeomorphisms.
\end{thm}

\begin{remark}
\label{rem-0}
\normalfont{
The functional invariants obtained in the theorem give a natural parametrisation of the orbit space:
\[M=\mbox{Diff}(2n)/\mbox{Symp}(2n),\]
where we denote by $\mbox{Diff}(2n)$ the group of diffeomorphisms of $\mathbb{R}^{2n}$ and by $\mbox{Symp}(2n)$ the subgroup of symplectomorphisms of the Darboux normal form $\omega=(\ref{Darboux})$. To compute their exact number, one has to write each function $\widetilde{P}_i$, $\widetilde{Q}_i$ as a sum of the generators of the corresponding ideal specified in (\ref{Q}) with coefficients functions $P_{ij}$, $Q_{ij}$. So for example, $\widetilde{Q}_1(p,q)=q_1Q_{11}(p,q)$, $Q_{11}(0)\ne 0$, $\widetilde{P}_2(p,q)=p_2+q_1P_{21}(p,q)+p_1P_{22}(p,q)$ and so on. In that way we obtain in total $n(2n-1)$ functional invariants $P_{ij}$, $Q_{ij}$ which are functions of $2n$-variables. If we denote by $M_k=J^k\mbox{Diff}(2n)/\mbox{Symp}(2n)$ the orbit space of $k$-jets of diffeomorphisms under the action of the symplectomorphism group, it immediately follows that the Poincar\'e series\footnote{recall that the Poincar\'e series of a graded vector space $V=\oplus_{k\geq 0}V_k$ is the formal series $P_V(t)=\sum_{k=0}^{\infty}\dim V_kt^k$.}  of $M$:
\[P_M(t)=\dim M_0+\sum_{k=1}^{\infty}(\dim M_k-\dim M_{k-1})t^k,\]
is the rational function:
\[P_M(t)=t\frac{n(2n-1)}{(1-t)^{2n}}.\]
}
\end{remark}
 
The nature of this parametrisation can be explained heuristically, without reference to Theorem \ref{thm-1},  as follows: by Darboux's theorem the group $\mbox{Diff}(2n)$ acts transitively on the space $\Omega^2_{S}(2n)$ of all symplectic structures, while $\mbox{Symp}(2n)$ is the isotropy subgroup of exactly one of them, namely of $\omega=(\ref{Darboux})$. Thus, the orbit space $M$ should be parametrised by the space  $\Omega^2_{S}(2n)$ of all symplectic structures, which becomes in that way an infinite dimensional homogeneous space. 

Theorem \ref{thm-1} makes this statement precise: indeed, the problem of classification of diffeomorphisms $\Phi$ by symplectomorphisms of $\omega=(\ref{Darboux})$ is equivalent to the problem of classification of pairs $(\omega,\Phi)$ under the whole group of diffeomorphisms, which is in turn equivalent to the classification of symplectic structures under diffeomorpshisms preserving the identity mapping $\Phi=Id$. The latter are no other than the identity mapping itself and thus the functional invariants are exactly the functions which parametrise the space $\Omega^2_S(2n)$ of all symplectic structures:
\begin{cor}
\label{cor-1}
Any symplectic form $\omega$ in $\mathbb{R}^{2n}$ can be written as:
\begin{equation}
\label{par-symp}
\omega=dp_1\wedge d\overline{Q}_1+\sum_{i=2}^nd(p_i+\overline{P}_i)\wedge d\overline{Q}_i,
\end{equation}
where $\partial_{q_i}\overline{Q}_i(0)\ne 0$, $i=1,\cdots,n$ and the $(2n-1)$-functions of $2n$-variables $\{\overline{Q}_i(p,q)\}_{i=1}^n$, $\{\overline{P}_i(p,q)\}_{i=2}^n$, belong to the finitely generated ideals:
\begin{equation}
\label{omega}
\{\overline{Q}_i\in I^{\omega}_{2i-1}\}_{i=1}^n, \quad \{\overline{P}_i\in I^{\omega}_{2i-2}\}_{i=2}^n.
\end{equation}
In particular, the Poincar\'e series $P_S(t)$ of the space of all symplectic structures $\Omega_S^{2}(2n)$ is:
\[P_S(t)=\frac{n(2n-1)}{(1-t)^{2n}}.\]
\end{cor}
\begin{proof}
The inverse diffeomorphism $\Phi^{-1}$ of Theorem \ref{thm-1} brings $\Phi$ to the identity and sends the Darboux normal form $\omega=(\ref{Darboux})$ to the desired form $\omega=(\ref{par-symp})$ for appropriate functions $\overline{Q}_i$, $\overline{P}_i$. Since the functions $\widetilde{Q}_i$, $\widetilde{P}_i$ satisfy (\ref{Q}), the functions $\overline{Q}_i$, $\overline{P}_i$ obviously satisfy (\ref{omega}). The same argument following Theorem \ref{thm-1} gives also the formula for the Poincar\'e series.
\end{proof}

\begin{remark}
\label{rem-1}
\normalfont{The parametrisation of the space $\Omega^2_S(2n)$  by $(2n-1)$-functions of $2n$-variables in the ideals $\{I_i^{\omega}\}_{i=1}^{2n-1}$ can in fact be obtained without reference to Theorem \ref{thm-1}; it follows from the observation that a primitive 1-form $a$ of a symplectic form $\omega$ is uniquely defined by $\omega$ modulo the choice of the differential of some function $h$ (since $\omega=da=d(a+dh)$, which can be chosen in such a way, so as to reduce the coefficients of the 1-form $a$ in the prescribed ideals. What is less trivial though is to choose the coordinate expression of the form $a$ and the function $h$ correctly, so as to obtain the desired expression (\ref{par-symp}) for $\omega$, which contains explicitly the normal form of the diffeomorphism $\Phi$ sending $\omega$ to its standard Darboux normal form (\ref{Darboux}).  
In any case, using Theorem \ref{thm-1} and its Corollary \ref{cor-1} above, we immediately obtain the formula:
\[P_M(t)=tP_S(t),\]
which is the infinite dimensional analog of the well known fact from representation theory:
\[\dim \mbox{GL}(2n)/\mbox{Sp}(2n)=\dim J^0\Omega^2_S(2n)=n(2n-1).\]
}
\end{remark}

\subsection{First Occurring Singularities of Hamiltonian Systems with One-Sided Constraints}

In this section we will work in $\mathbb{R}^{2n+2}$, $n\geq 1$, for notational reasons. The planar case $n=0$ has been treated extensively in \cite{Kir},  \cite{K} and will only be discussed here in Remark \ref{rem-2D} below. We identify Hamiltonian systems with constraints with triples $(\omega,f,H)$ where $\omega$ is again a symplectic form, $f$ is a non-singular function, $f(0)=0$, $df(0)\ne 0$, and $H=\{h=0\}$ is a smooth hypersurface, $h(0)=0$, $dh(0)\ne 0$. Fix the Darboux  normal form of $\omega$: 
\begin{equation}
\label{Darboux-1}
\omega=dx\wedge dy+dp\wedge dq.
\end{equation}
The purpose is to classify first occurring singularities of pairs $(f,H)$ by symplectomorphisms of $\omega=(\ref{Darboux-1})$. These singularities are distinguished by the same conditions as for the corresponding pairs of glancing hypersurfaces $F=\{f=0\}$, $H=\{h=0\}$ considered by Melrose in \cite{M}:

\begin{equation}
\label{glancing-1}
df\wedge dh(0)\ne 0 \quad (\mbox{for $n\geq 1$}),
\end{equation}
\begin{equation}
\label{glancing-2}
\{f,h\}(0)=0,\quad \{f,\{f,h\}\}(0)\ne 0,\quad \{h,\{f,h\}\}(0)\ne 0.
\end{equation}
We denote by $S_1$ their singularity class.

The main result in \cite{M} is that in the $C^{\infty}$-category, any pair of glancing hypersurfaces can be reduced by a symplectomorphism of $\omega=(\ref{Darboux-1})$ to the simple normal form:
\begin{equation}
\label{nf-M}
F=\{y=0\}, \quad H=\{x^2+y+p_1=0\}.
\end{equation}
In the analytic category, classification of glancing hypersurfaces is well known to contain functional moduli (c.f. \cite{Vor}). 

In the same paper \cite{M} the author noticed that replacing one of the hypersurfaces, say $F=\{f=0\}$, by its defining function $f$, makes the classification problem substantially more difficult: moduli occur already from the classification of 2-jets. The main result in this direction is the following:

\begin{thm}
\label{thm-2}
In the space of $2n$-jets of singular pairs $(f,H)\in S_1$ in $(\mathbb{R}^{2n+2},\omega)$, $n\geq 1$, there exists an open set $U$ such that any pair with $j^{2n}(f,H)\in U$ can be reduced by a symplectomorphism of $\omega=(\ref{Darboux-1})$ to the normal form:
\begin{equation}
\label{thm-nf-fH}
f= y,  \quad H=\{x^2+g(y,p,q)=0\}, 
\end{equation}
where
\begin{equation}
\label{thm-nf-R}
g(y,p,q)=r(y)+p_1+\sum_{i=2}^n(p_i+\widetilde{P}_i(p,q))y^{2i-2}+\sum_{i=1}^{2n-1}\widetilde{Q}_i(p,q)y^{2i-1}+\phi(y,p,q)y^{2n}=0\},
\end{equation}
with $r'(0)\ne 0$, $\phi(y,0,0)=0$, $\partial_{q_i}\widetilde{Q}_i(0)\ne 0$, $i=1,\cdots,n$, and the $(2n-1)$ functions $\{\widetilde{P}_i(p,q)\}_{i=2}^n$,  $\{\widetilde{Q}_i(p,q)\}_{i=1}^n$ belonging in the ideals:
\begin{equation}
\label{r-ideals}
\{\widetilde{P}_i\in I_{2i-2}\}_{i=2}^n, \quad \{\widetilde{Q}_i\in I_{2i-1}\}_{i=1}^n.
\end{equation}
The function of $1$-variable $r$, the function of $(2n+1)$-variables $\phi$, and the $(2n-1)$ functions of $2n$-variables $\{\widetilde{P}_i\}_{i=2}^{n}$, $\{\widetilde{Q}_i\}_{i=1}^n$ in the corresponding ideals (\ref{r-ideals}) above, are functional invariants. 
\end{thm}

\begin{remark}
\normalfont{
As in Remark \ref{rem-0} following Theorem \ref{thm-1}, the conditions (\ref{r-ideals}) define in total $n(2n-1)$ functional invariants of $2n$-variables $P_{ij}$, $Q_{ij}$ corresponding to the functions $\widetilde{P}_i$, $\widetilde{Q}_i$.
}
\end{remark}

\begin{remark}
\normalfont{
In \cite{KM} the authors have obtained the following exact normal form for the first occurring singularities of generic triples $(\omega,f,H)$:
\begin{equation}
\label{nf-KM}
\omega=dx\wedge d\widehat{f}(y,p,q)+\widetilde{\omega}, \quad f=\widehat{f}(y,p,q), \quad H=\{x^2+y=0\},
\end{equation}
where
\begin{equation}
\label{nf-KM-f}
\widehat{f}=\widehat{r}(y)+\sum_{i=1}^n(p_iy^{2i-2}+q_iy^{2i-1})+\psi(y,p,q)y^{2n},
\end{equation}
with the functional invariants $\widehat{r}(y)$, $\widehat{r}'(0)\ne 0$, $\psi(y,p,q)$, $\phi(y,0,0)=0$, and the symplectic form $\widetilde{\omega}$ in $\mathbb{R}^{2n}(p,q)$ being a functional invariant as well. 
To deduce this normal form  from normal form (\ref{thm-nf-fH})-(\ref{thm-nf-R}) of Theorem \ref{thm-2}, one has to perform the following operations: first consider the diffeomorphism $\Phi(p,q)=(p_1, \widetilde{Q}_1(p,q),\cdots,p_n+\widetilde{P}_n(p,q),\widetilde{Q}_n(p,q))$ in the symplectic space $(\mathbb{R}^{2n},dp\wedge dq)$. Using Corollary \ref{cor-1} of Theorem \ref{thm-1} we can bring $\Phi$ to the identity and send $dp\wedge dq$ to the form $\widetilde{\omega}$ (which moreover has the form $\widetilde{\omega}=(\ref{par-symp}))$. This gives the following normal form for the triple $(\omega,f,H)$:
\[\omega=dx\wedge dy+\widetilde{\omega}, \quad f=y,\]
\[H=\{x^2+r(y)+\sum_{i=1}^n(p_iy^{2i-2}+q_iy^{2i-1})+\phi(y,p,q)y^{2n}=0\}.\]
Since $r'(0)\ne 0$, the diffeomorphism:
\[r(y)+\sum_{i=1}^n(p_iy^{2i-2}+q_iy^{2i-1})+\phi(y,p,q)y^{2n}\mapsto y,\]
brings the above normal form to the desired normal form (\ref{nf-KM})-(\ref{nf-KM-f}).  
}
\end{remark}

\begin{remark}[2D-Case]
\label{rem-2D}
\normalfont{
Notice that restriction of the normal form $(\omega,f,H)=((\ref{Darboux-1}),(\ref{thm-nf-fH})-\ref{thm-nf-R})$ on the $(x,y)$-plane $p=q=0$ gives the normal form: 
\begin{equation}
\label{nf-RMD-new}
\omega=dx\wedge dy, \quad f=y, \quad H=\{x^2+r(y)=0\},
\end{equation}
with the function $r(y)$, $r'(0)\ne 0$. It is not difficult to see that this is indeed an exact normal form for first occurring singularities of pairs $(f,H)\in S_1$ on the plane $(\mathbb{R}^2,\omega)$, with the functional invariant $r(y)$: indeed, following the proof of Theorem \ref{thm-2} in Section 4, we can reduce the pair $(\omega,f)$ to its standard normal form $(dx\wedge dy,y)$ and the curve $H$ to the form $H=\{x^2+a(y)x+b(y)=0\}$, for some functions $a$, $b$,  vanishing at the origin and $b'(0)\ne 0$. If we denote by $C_{f,h}=\{\{f,h\}=0\}=\{2x+a(y)=0\}$ the critical curve\footnote{this critical hypersurface is invariantly associated to the triple $(\omega,f,H)$, and reducing it to normal form will also play an important role in the proof of Theorem \ref{thm-2}.}, we see that the diffeomorphism $x\mapsto x-a(y)/2$ preserves the pair $(dx\wedge dy,y)$, sends the critical curve to $C_{f,h}=\{x=0\}$ and the curve $H$ to the desired normal form $H=\{x^2+r(y)=0\}$, where $r'(0)\ne 0$. The fact that $r$ is a functional invariant follows now from the fact that any diffeomorphism of the triple $(\omega,f,C_{f,h})=(dx\wedge dy, y, \{x=0\})$ is necessarily the identity. 

Notice that since $r'(0)\ne 0$ the function $r$ defines a diffeomorphism on the line $\mathbb{R}$. If we denote by $\widehat{r}=r^{-1}$ the inverse diffeomorphism, then the change of coordinates $y\mapsto \widehat{r}(y)$ brings the normal form (\ref{nf-RMD-new}) to the normal form:
\begin{equation}
\label{nf-RMD-KM}
\omega=dx\wedge d\widehat{r}(y), \quad f=\widehat{r}(y), \quad H=\{x^2+y=0\},
\end{equation} 
which coincides with the normal form (\ref{nf-KM})-(\ref{nf-KM-f}) restricted on the $(x,y)$-plane $p=q=0$. This normal form was also obtained in \cite{KM}.  

The normal forms obtained here differ from the classical normal form of the relative isochore Morse lemma in \cite{Kir}, \cite{K}:
\begin{equation}
\label{isochore}
\omega=dx\wedge dy, \quad f=\phi(y+x^2), \quad H=\{y=0\}, \quad \phi'(0)\ne 0,
\end{equation} 
but their proof is much shorter. We leave to the reader the (rather not so easy) exercise to find the diffeomorphism sending the normal form (\ref{nf-RMD-KM}) to the classical normal form (\ref{isochore}) above (hint: one may use the construction in \cite{F} for the ordinary isochore Morse lemma, used also in \cite{K} for the relative case). 
}
 \end{remark}


\section{Proof of Thorem \ref{thm-1}} 
\subsection{Auxiliary Lemmas: Non-Singular Hamiltonian Systems with One-Sided Constraints}

Here we work in symplectic space $(\mathbb{R}^{2n},\omega)$.  We identify Hamiltonian systems with pairs $(\omega,P)$, where $P$ is a function (the Hamiltonian), $P(0)=0$ . Let now $\mathcal{Q}$ be a smooth hypersurface in the symplectic space $(\mathbb{R}^{2n},\omega,P)$. We say that the pair $(P,\mathcal{Q})$ is non-singular if the hypersurface $\mathcal{Q}$ is transversal to the Hamiltonian vector field $Z_P$ of $P$. This implies that for any function $Q$, $dQ(0)\ne 0$, defining $\mathcal{Q}=\{Q=0\}$, the following condition holds: 
\[\{P,Q\}(0)\ne 0.\]   
The triple $(\omega,P,\mathcal{Q})$ can be viewed (in terms of the previous section) as defining a non-singular Hamiltonian system $(\omega,P)$ with one-sided constraints $\mathcal{Q}$. 

\medskip

\noindent NOTATION. Here and below we denote by $g( \widehat{\cdot} )$ any function (map) $g$ which does not depend on the coordinates which are under the  `` $\widehat{}$ '' symbol.

\begin{lem}[\cite{A1},\cite{KM},\cite{M}]
\label{lem-1}
Any non-singular  pair $(P,\mathcal{Q})$ in the symplectic space $(\mathbb{R}^{2n},\omega)$ can be reduced, by symplectomorphisms of the Darboux normal form $\omega=(\ref{Darboux})$, to the normal form:
\begin{equation}
\label{nf-S0}
\quad P=p_1, \quad \mathcal{Q}=\{q_1=0\}.
\end{equation}
\end{lem}
\begin{proof}
Normal form (\ref{nf-S0}) is standard and it is a consequence of (the proof of) Darboux's theorem.  For coherence, we sketch here a short proof, distilled from \cite{KM}, which will also be used later in the proof of Theorem \ref{thm-2}. It relies on normalising simultaneously the whole triple $(\omega,P,\mathcal{Q})$ (where $\omega$ is not yet in Darboux normal form (\ref{Darboux})). Since the Hamiltonian vector field $Z_P$ is transversal to the hypersurface $\mathcal{Q}$, we can bring the pair $(Z_P,\mathcal{Q})$ to the standard normal form $(\partial_{q_1},\{q_1=0\})$. Since $P$ is a first integral of $Z_P=\partial_{q_1}$ we have that $P=P(\widehat{q}_1)$, i.e. $P$ is independent of the coordinate $q_1$. The normal form of $Z_P$ implies, from the equation $Z_P\lrcorner \omega=dP$, that $\omega$ can be reduced to the preliminary normal form:
\[\omega=dP(\widehat{q}_1)\wedge dq_1+\widehat{\omega},\]
where $\widehat{\omega}$ is a 2-form such that $Z_{q_1}\lrcorner \widehat{\omega}=0$. Since $\omega$ is closed $\widehat{\omega}$ is also closed and since $\omega^{n}(0)\ne 0$, we obtain that $dP\wedge \widehat{\omega}^{n-1}(0)\ne 0$. In particular $\widehat{\omega}^{n-1}(0)\ne 0$ and thus the 2-form $\widehat{\omega}$ is a quasi-symplectic form\footnote{i.e. a closed 2-form of maximal rank in odd-dimensional space, c.f. Section 4.} in $\mathbb{R}^{2n-1}(\widehat{q}_1)$. Since $dP\wedge \widehat{\omega}^{n-1}(0)\ne 0$, using an odd-dimensional version of the Darboux theorem (c.f. \cite{Z}), we can bring the pair $(P(\widehat{q}_1),\widehat{\omega})$ to the normal form $(p_1,\sum_{i=2}^{n}dp_2\wedge dq_2)$ as in the proof of Lemma \ref{lem-3} in Section 4 below. This finishes the proof.
\end{proof}

\medskip

\begin{lem}
\label{lem-2}
Any symplectomorphism of $\omega=(\ref{Darboux})$ which preserves also the pair $(P,\mathcal{Q})=(\ref{nf-S0})$ is of the form:
\begin{equation}
\label{s1}
(p,q)\mapsto (p_1,q_1,B(\widehat{p}_1,\widehat{q}_1)),
\end{equation}
where the map:
\[B(\widehat{p}_1,\widehat{q}_1):=(B_1(\widehat{p}_1,\widehat{q}_1),\cdots, B_{2n-2}(\widehat{p}_1,\widehat{q}_1)),\]
is a symplectomorphism of the restriction $\omega|_{p_1=q_1=0}=\sum_{i=2}^ndp_i\wedge dq_i$  of $\omega$ on the $(2n-2)$-dimensional symplectic space $p_1=q_1=0$:
\begin{equation}
\label{ps1}
B^*\omega|_{q_1=p_1=0}=\omega|_{p_1=q_1=0}.
\end{equation}
\end{lem}
\begin{proof}
Consider symplectomorphisms of the pair $(\omega,p_1)$, where $\omega=(\ref{Darboux})$. Any such symplectomorphism preserves also the Hamiltonian vector field $Z_{p_1}=\partial_{q_1}$ of the function $p_1$ and it is thus of the form:
\begin{equation}
\label{s2}
(p,q)\mapsto (p_1,q_1+A(\widehat{q}_1),B(\widehat{q}_1)),
\end{equation}
for an appropriate function $A(\widehat{q}_1)$ and an appropriate map 
\[B(\hat{q}_1):=(B_1(\widehat{q}_1),\cdots, B_{2n-2}(\widehat{q}_1)).\] 
Now, the requirement that the symplectomorphism (\ref{s2}) preserves $\mathcal{Q}=\{q_1=0\}$, or what is equivalent, the ideal $I^{\omega}_1=<q_1>$, implies that $A(\widehat{q}_1)\equiv 0$ and hence, it preserves the coordinate  $q_1$ as well. From this it follows that the symplectomorphism (\ref{s2}) preserves also the Hamiltonian vector field $Z_{q_1}=\partial_{p_1}$ of $q_1$ and it is thus of the form (\ref{s1}):
\[(p,q)\mapsto (p_1,q_1,B(\widehat{p}_1,\widehat{q}_1)).\]
Since this map is a symplectomorphism of $\omega$ it immediately follows that the map $B(\widehat{p}_1,\widehat{q}_1)$ is a symplectomorphism of the restriction $\omega|_{p_1=q_1=0}=\sum_{i=2}^ndp_i\wedge dq_i$, i.e. equation (\ref{ps1}) holds.
\end{proof}

\subsection{Proof of Theorem \ref{thm-1}}

We fix the Darboux normal form $\omega=(\ref{Darboux})$ and we normalise the diffeomorphism $\Phi$ with fixed conditions (\ref{con-2}). For $n=1$, Theorem \ref{thm-1} is an immediate corollary of Lemmas \ref{lem-1} and \ref{lem-2}. For $n\geq 2$, we 
apply first Lemma \ref{lem-1} to the pair $(P_1,Q_1)$ and we obtain the preliminary normal form:
\[P_1=p_1, \quad Q_1\in I_1^{\omega}, \quad \partial_{q_1}Q_1(0)\ne 0.\]
The next step is to normalise the pair of functions $(P_2,Q_2)$ by symplectomorphisms  of the triple $(\omega, p_1,I^{\omega}_1)$, which, by Lemma \ref{lem-2} are mappings of the form (\ref{s1}):
\[(p,q)\mapsto (p_1,q_1,B(\widehat{p}_1,\widehat{q}_1)).\] 
Notice that any such symplectomorphism sends the function $Q_1$ to some new function $\widetilde{Q}_1$, but still in the ideal $I^{\omega}_1$. For simplicity we denote this new function by $Q_1$ as well.  To perform the normalisation restrict the pair $(P_2,Q_2)$ to the symplectic subspace $p_1=q_1=0$ and apply again Lemma \ref{lem-1}. We obtain the normal form for the restricted pair:
\[P_2|_{p_1=q_1=0}=p_2, \quad Q_2|_{p_1=q_1=0}\in <q_2>, \quad \partial_{q_2}Q_2(0)\ne 0,\]
by a symplectomorphism $B(\widehat{p}_1,\widehat{q}_1)$ of $\omega|_{p_1=q_1=0}$.  It follows that the pair $(P_2,Q_2)$ can be reduced, by a symplectomorphism of the triple $(\omega,p_1,I^{\omega}_1)$, to the normal form:
\[P_2=p_2 \mbox{mod} I^{\omega}_2, \quad Q_2\in I^{\omega}_3, \quad \partial_{q_2}Q_2(0)\ne 0.\]
Now  Lemma \ref{lem-2} applied to the triple $(\omega|_{p_1=q_1=0},p_2|_{p_1=q_1=0},I^{\omega}_3|_{p_1=q_1=0})$, implies that any symplectomorphism of the whole tuple $(\omega, p_1,I^{\omega}_1,p_2 \mbox{mod}I^{\omega}_2,I^{\omega}_3)$ is necessarily of the form:
\[(p,q)\mapsto (p_1,q_1,p_2,q_2,B(\widehat{p}_1,\widehat{q}_1,\widehat{p}_2,\widehat{q}_2)),\]
for some symplectomorphism
\[B(\widehat{p}_1,\widehat{q}_1,\widehat{p}_2,\widehat{q}_2)=(B_1(\widehat{p}_1,\widehat{q}_1,\widehat{p}_2,\widehat{q}_2),\cdots,B_{2n-4}(\widehat{p}_1,\widehat{q}_1,\widehat{p}_2,\widehat{q}_2))\]
of the restriction $\omega|_{p_1=q_1=p_2=q_2=0}=\sum_{i=3}^ndp_i\wedge dq_i$. This also proves the theorem for $n=2$. For $n\geq 3$, we continue in the same way as before and we arrive, after $(n-2)$ more consecutive steps of restrictions and normalisations obtained by applying Lemmas \ref{lem-1} and \ref{lem-2}, to the required normal form:
\[P_1=p_1,\]
\[Q_1\in I^{\omega}_1, \quad \partial_{q_1}Q_1(0)\ne 0, \]
\[P_2=p_2 \mbox{mod} I^{\omega}_2, \]
\[Q_2\in I^{\omega}_3, \quad \partial_{q_2}Q_2(0)\ne 0,\]
\[\vdots\]
\[P_n=p_n \mbox{mod}I^{\omega}_{2n-2},\]
\[Q_n\in I^{\omega}_{2n-1}, \quad \partial_{q_n}Q_n(0)\ne 0.\]
Successive applications of Lemma \ref{lem-2} in each step of the previous normalisations and restrictions,  implies that any symplectomorphism of $\omega=(\ref{Darboux})$ which preserves the whole tuple 
\[(p_1,I^{\omega}_1,\cdots,p_n\mbox{mod}I^{\omega}_{2n-2},I^{\omega}_{2n-1})\] 
is necessarily the identity, and the theorem is proved.

\qed

\section{Proof of Theorem \ref{thm-2}}
\subsection{Auxiliary Lemmas: Pairs of Functions in Quasi-\newline
Symplectic Space}
A quasi-symplectic space is the odd-dimensional analog of a symplectic space, i.e. a space $\mathbb{R}^{2n+1}$ endowed with a closed 2-form $\widehat{\omega}$ of maximal rank $2n$ (usually called a quasi-symplectic structure). The 1-dimensional (line) field of kernels of the form $\widehat{\omega}$ foliates the space $\mathbb{R}^{2n+1}$ in 1-dimensional curves and thus any quasi-symplectic space is fibered over a symplectic space $\mathbb{R}^{2n}$ (the base of the foliaton) with 1-dimensional fibers. More precisely there exists a projection $\pi:(\mathbb{R}^{2n+1},\omega)\rightarrow (\mathbb{R}^{2n},\widetilde{\omega})$ where $\widetilde{\omega}$ is a symplectic form such that $\widehat{\omega}=\pi^*\widetilde{\omega}$.  By the odd-dimensional analog of Darboux's theorem (c.f. \cite{Z}) all quasi-symplectic spaces are locally equivalent, i.e. there exists coordinates $(y,p,q):=(y,p_1,q_1,\cdots,p_n,q_n)$ such that $\widehat{\omega}$ is reduced to the same Darboux normal form $\widehat{\omega}=(\ref{Darboux})$.  Notice that in these coordinates $\pi(y,p,q)=(p,q)$ and also $\widetilde{\omega}=(\ref{Darboux})$.

Consider now a pair of functions $(f,g)$ in quasi-symplectic space $(\mathbb{R}^{2n+1},\widehat{\omega})$. The problem is to classify the pair $(f,g)$ under quasi-symplectomorphisms, i.e. diffeomorphisms of the quasi-symplectic Darboux normal form $\widehat{\omega}=(\ref{Darboux})$.  We consider only the non-singular case, where the pair $(f,g)$ is in general position. By this we mean that both of the functions are non-singular, they are transversal to each other, and are also transversal to the kernel field of $\widehat{\omega}$:
\[ df(0)\ne 0, \quad dg(0)\ne 0,\quad df\wedge dg(0)\ne 0\]
\[df\wedge \widehat{\omega}^n(0)\ne 0, \quad dg\wedge \widehat{\omega}^n(0)\ne 0.\]

 \begin{lem}
 \label{lem-3}
 In the space of $2n$-jets of non-singular pairs $(f,g)$ there exists an open set $U$ such that any pair $(f,g)$ with $j^{2n}(f,g)\in U$ is equivalent, under quasi-symplectomorphisms of $\widehat{\omega}=(\ref{Darboux})$, to the normal form
 \begin{equation}
 \label{nf-pair}
 f=y,\quad g=r(y)+p_1+\sum_{i=2}^{n}(p_i+\widetilde{P}_i(p,q))y^{2i-2}+\sum_{i=1}^n\widetilde{Q}_i(p,q)y^{2i-1}+\phi(y,p,q)y^{2n},
 \end{equation}
 where $r'(0)\ne 0$, $\phi(y,0,0)=0$, and the functions $\{\widetilde{P}_i(p,q)\}_{i=2}^n$, $\{\widetilde{Q}_i(p,q)\}_{i=1}^n$, belong to the finitely generated ideals:
\begin{equation}
\label{r-ideals-1} 
\{\widetilde{P}_i\in I_{2i-2}\}_{i=2}^n, \quad \{\widetilde{Q}_i\in I_{2i-1}\}_{i=1}^n.
\end{equation}
The function of $1$-variable $r$, the function of $(2n+1)$-variables $\phi$, and the $(2n-1)$ functions of $2n$-variables $\{\widetilde{P}_i\}_{i=2}^{n}$, $\{\widetilde{Q}_i\}_{i=1}^n$ in the corresponding ideals (\ref{r-ideals-1}) above, are functional invariants.

 \end{lem}
 \begin{proof}
We normalise first $f$ by quasi-symplectomorphisms of $\widehat{\omega}=(\ref{Darboux})$. These are maps of the form (c.f. \cite{Z})
\[(y,p,q)\mapsto(Y(y,p,q),A(p,q),B(p,q)),\]
where $Y(y,p,q)$ is a function such that $\partial_yY(0)\ne 0$ and the map $(A(p,q),B(p,q)):=(A_1(p,q),B_1(p,q),\cdots,A_n(p,q),B_n(p,q))$ is a symplectomorphism of $\widetilde{\omega}=(\ref{Darboux})$. Since the condition $df\wedge \widehat{\omega}^n(0)\ne 0$ is equivalent to $\partial_yf(0)\ne 0$ we can send $f$, by an appropriate choice of the function $Y(y,p,q)$ to the normal form $f=y$. Thus it remains to normalise the function $g$ by quasi-symplectomorphisms of $\widehat{\omega}=(\ref{Darboux})$ which also preserve $y$, i.e. of the form:
\[(y,p,q)\mapsto (y,A(p,q),B(p,q)).\]
By division with the ideal $<p,q>$ we can write $g=r(y)+G(y,p,q)$, for some function $r(y)=g|_{p=q=0}$, with $r'(0)\ne 0$ (because $dg\wedge \widehat{\omega}^n(0)\ne 0\Leftrightarrow \partial_yg(0)\ne 0$) and some function $G$ such that $G(y,0,0)=0$. We now expand $G$ as 
\[G=P_1(p,q)+Q_1(p,q)y+P_2(p,q)y^2+Q_2(p,q)y^3+\cdots .\]
 Since $df\wedge dg(0)\ne 0\Leftrightarrow dy\wedge dg(0)\ne 0$ we can assume that $\partial_{p_1}P_1(0)\ne 0$. For a generic germ $g$ we can suppose that the functions $P_1, Q_1, \cdots, P_{n}, Q_n$ are differentially independent:
\[dP_1\wedge dQ_1\wedge \cdots \wedge dP_n\wedge dQ_n(0)\ne 0,\]
a condition which defines the open set $U$ in the statement of the theorem.  This implies that the corresponding map $\Phi(p,q)=(P_1(p,q),\cdots,Q_n(p,q))$ defines a diffeomorphism in the symplectic space $(\mathbb{R}^{2n},\widetilde{\omega}=(\ref{Darboux}))$. Thus the problem to obtain exact normal form for $g$ under quasi-symplectomorphisms of the pair $(y,\widehat{\omega}=(\ref{Darboux}))$, reduces to the exact classification of the diffeomorphism $\Phi$ under symplectomorphisms of $\widetilde{\omega}=(\ref{Darboux})$. The rest of the proof follows now from Theorem \ref{thm-1}.
 \end{proof}
 
 \subsection{Proof of Theorem \ref{thm-2}}

The proof of Theorem \ref{thm-2} is analogous to the proof of Lemma \ref{lem-1}: we start with a whole triple $(\omega,f,H)$ which we will bring simultaneously to the desired normal form. Since $Z_f$ is non-singular we can rectify it to $Z_f=\partial_x$. Since it is Hamiltonian for $f$ we obtain $\partial_xf=0\Leftrightarrow f=\widehat{f}(y,p,q)$ for some function $\widehat{f}$ defined on $\mathbb{R}^{2n+1}(y,p,q)$. We write $f$ instead of $\widehat{f}$ (meaning that $f=f(y,p,q)$ is independent of the variable $x$). Since $\{f,h\}(0)=0\Leftrightarrow \partial_xh(0)=0$ and $\{f,\{f,h\}\}(0)\ne 0\Leftrightarrow \partial^2_{x^2}h(0)\ne 0$, we obtain $H=\{x^2+a(y,p,q)x+b(y,p,q)=0\}$ for some functions $a$, $b$ defined in $\mathbb{R}^{2n+1}(y,p,q)$, vanishing at the origin and such that $db(0)\ne 0$ (by the non-singularity of $H$). The change of coordinates $x\mapsto x-a(y,p,q)/2$ preserves the Hamiltonian vector field $\partial_x$ and sends $H$ to $H=\{x^2+g(y,p,q)=0\}$ for some function $g$ defined on $\mathbb{R}^{2n+1}(y,p,q)$, vanishing at the origin and such that $dg(0)\ne 0$.  Denote by 
\[C_{f,h}=\{\{f,h\}=0\}\]
the critical hypersurface of the pair $(f,H)$ (it is invariantly associated to the triple $(\omega,f,H)$). In the coordinates above it is also reduced to normal form 
\[C_{f,h}=\{x=0\}.\] 
Now, as in the proof of Lemma \ref{lem-1}, the normal form $\partial_x$ of $Z_f$ implies, by equation $Z_f\lrcorner \omega=df$,
that $\omega$ can be written as:
\[\omega=df\wedge dx+\widehat{\omega},\]
where $\widehat{\omega}$ is a quasi-symplectic form in $\mathbb{R}^{2n+1}(y,p,q)$ (by the same reasoning as in Lemma \ref{lem-1}). Thus we have reduced the triple $(\omega,f,H)$ to the preliminary normal form:
\begin{equation}
\label{pnf}
\omega=df\wedge dx+\widehat{\omega},\quad f=f(y,p,q), \quad H=\{x^2+g(y,p,q)=0\}
\end{equation}
where the pair $(Z_f,C_{f,h})$ is in simple normal form $(\partial_x,\{x=0\})$. To continue the normalisation further we consider diffeomorphisms preserving the normal form $(\partial_x,\{x=0\})$ of the pair $(Z_f,C_{f,h})$. Any such diffeomorphism preserves also the coordinate $x$ and it is thus of the form:
\[(x,y,p,q)\mapsto (x,A(y,p,q))\] 
for some diffeomorphism $A(y,p,q):=(A_1(y,p,q),\cdots,A_{2n+1}(y,p,q))$ of $\mathbb{R}^{2n+1}$. Hence the problem to obtain exact normal form reduces to the classification of pairs of functions $(f,g)$ in the quasi-symplectic space $(\mathbb{R}^{2n+1}(y,p,q),\widehat{\omega})$. By the transversality of $f$ with $H$ we obtain that $df\wedge dg(0)\ne 0$, by the non-degeneracy of $\omega$ we obtain $df\wedge \widehat{\omega}^{n}(0)\ne 0$, and finally by the condition $\{h,\{f,h\}(0)\ne 0$ we obtain $dg\wedge \widehat{\omega}^n(0)\ne 0$. Thus the pair $(f,g)$ is a non-singular and the rest of the proof follows from Lemma \ref{lem-3}. The final normal form announced in the theorem is obtained by the change $x\mapsto -x$.

\qed

\section*{Acknowledgements}
This research has been supported by the Research Foundation of S\~ao Paulo (FAPESP), grand No: 2017/23555-19.


\begin{thebibliography}{55}
\bibitem{A0} Arnol'd V. I., Koslov V. V., Neishtadt A. I, \textit{Mathematical Aspects of Classical and Celestial Mechanics}, Springer-Verlag, (2000)
\bibitem{A1} Arnol'd V. I., \textit{Lagrangian Manifolds with Singularities, Asymptotic Rays and the Open Swallowtail}, Moscow State University. Translated from Funktsional'nyi Analiz i Ego Prilozheniya, 15(4), (1981), 1-14
\bibitem{A2}  Arnol'd V. I., \textit{Singularities in Variational Calculus}, J. Math. Sci., (1984), 27: 2679,
\bibitem{A3} Arnol'd V. I., Givental' A. B., \textit{Symplectic Geometry and its Applications}, Encyclopaedia of Mathematical Sciences, 4, Springer-Verlag, (1990)
\bibitem{A4} V. I. Arnol'd, \textit{Singularities of Caustics and Wave Fronts}, Mathematics and its Applications, Springer-
Science+Media, B. V., (1990)
\bibitem{A6} Arnol'd V. I., \textit{Mathematical Problems in Classical Physics}, in: Sirovich L. (eds) Trends and Perspectives in Applied Mathematics. Applied Mathematical Sciences, Springer, 100, (1994), 1-20
\bibitem{Dir} Dirac P. M., \textit{Generalized Hamiltonian Dynamics}, Canad. J. Math. 2,(1950), 129-148
\bibitem{D}  Dubrovskiy S., \textit{Moduli of Fedosov Structures}, Ann. Glob. Anal. Geom., 27, (2005), 273-297 
\bibitem{F} Fran\c{c}oise J. -P. , \textit{Relative Cohomology and Volume Forms}, Singularities, Banach Center Publications, 20, (1988), 207-222
\bibitem{G} Garay M. D., \textit{An Isochore Versal Deformation Theorem}, Topology, 43 (2004) 1081-1088
\bibitem{Kir} Kirillov I., \textit{Morse-Darboux Lemma on Surfaces with Boundary}, Journal of Geometry and Physics, 129, (2018), 34-40
\bibitem{K} Kourliouros K., \textit{Gauss-Manin Connections for Boundary Singularities and Isochore Deformations}, Dem. Mat, 48, 2, (2015), 250-288 
\bibitem{KM} Kourliouros K., Zhitomirskii M., \textit{First Occurring Singularities of Functions in Symplectic Semi-Space}, J. Pur. Appl. Math., 2(2), (2018), 1-3.
\bibitem{M} Melrose R. B., \textit{On Equivalence of Glancing Hypersurfaces I,} Inventiones Math., 37, (1976), 165-191
\bibitem{CV} Colin de Verdi\`ere Y., Vey J., \textit{Le Lemme de Morse Isochore}, Topology, 18, (1979), 283-293
\bibitem{Ve} Vey J., \textit{Sur le Lemme de Morse}, Inventiones math. 40,  (1977), 1-9
\bibitem{Vor} Voronin S. M., \textit{The Darboux-Whitney Theorem and Related Questions}, Adv. Soviet Math., 14, Amer. Math. Soc., Providence, RI, (1993), 139–233
\bibitem{Z} Zhitomirskii M., \textit{Typical Singularities of Differential 1-Forms and Pfaffian Equations}, Amer. Math. Soc. (1992)
\end{thebibliography}
\end{document}